\documentclass[review,hidelinks,onefignum,onetabnum]{siamart251216}

\usepackage{graphicx}
\usepackage{amssymb}
\usepackage{mathrsfs}
\usepackage{color}
\usepackage{ulem}
\theoremstyle{plain}
\usepackage{hyperref}

\nolinenumbers

\newcommand{\norm}[2]{\left\|#1\right\|_{#2}}
\newcommand{\scalar}[2]{\langle{ #1},{#2} \rangle}

\newcommand{\set}[1]{\left\{#1\right\}}
\newcommand{\abs}[1]{\left\lvert #1 \right\rvert}
\newcommand{\lr}[1]{\left(#1\right)}
\newcommand{\nat}{\mathbb N}
\newcommand{\real}{\mathbb R}

\DeclareMathOperator{\sgn}{sgn}
\newcommand{\domain}{\mathcal D}
\newcommand{\range}{\mathcal R}
\newcommand{\Bsp}[1]{B^{#1}(\real^d)}
\newcommand{\Bspo}[1]{B^{#1}(\Omega)}
\newcommand{\Sp}{\mathcal S^{\prime}(\real^d)}
\newcommand{\ad}{I - \Delta}
\newcommand{\oned}{L}
\newcommand{\pdo}{$\Psi$DO}
\newcommand{\pdos}{$\Psi$DOs}
\newcommand{\F}{{\mathcal F}}
\newcommand{\yd}{y^{\delta}}

\newcommand{\uast}{u^\ast}
\newcommand{\udag}{u^\dag}
\newcommand{\ud}{u^\delta}

\newcommand{\taw}{\mathcal T_{\alpha,W}}

\newtheorem{prop}{Proposition}[section]
\newtheorem{remark}{Remark}
\newtheorem{ass}{Assumption}
\newtheorem{xmpl}{Example}

\ifpdf
\hypersetup{
  pdftitle={Interpolation and inverse problems in spectral Barron spaces},
  pdfauthor={Shuai Lu and Peter Mathe}
}
\fi

\headers{Interpolation and inverse problems in spectral Barron spaces}{Shuai Lu and Peter Math\'e}

\title{Interpolation and inverse problems in spectral Barron spaces\thanks{Submitted to the editors \today.
\funding{This work is supported by National Key Research and Development Programs of China (No. 2023YFA1009103), NSFC (No. 92570106) and Science and Technology Commission of Shanghai Municipality (No. 23JC1400501).}}}

\author{Shuai Lu\thanks{Corresponding author. School of Mathematics Sciences, SKLCAM and LMNS, Fudan University, Shanghai, China
  (\email{slu@fudan.edu.cn})}
\and Peter Math\'e \thanks{Weierstrass Institute for Applied Analysis and Stochastics, Mohrenstrasse 39, Berlin 10117, Germany
  (\email{peter.mathe@wias-berlin.de})}
}

\begin{document}

\maketitle
\begin{abstract}
Spectral Barron spaces, which quantify the absolute value of weighted Fourier coefficients of a function, have gained considerable attention due to their capability for universal approximation across certain function classes. By establishing a connection between these spaces and a specific positive linear operator, we investigate the interpolation and scaling relationships among diverse spectral Barron spaces. Furthermore, we introduce a link condition by relating the spectral Barron space to inverse problems, illustrating this with three exemplary cases. We revisit the notion of universal approximation within the context of spectral Barron spaces and validate an error bound for Tikhonov regularization, penalized by the spectral Barron norm.
\end{abstract}


\section{Introduction}
\label{sec:intro}
In recent years, neural networks have undergone an explosion of development within the field of applied and computational mathematics, pushing the boundaries of what was previously thought impossible \cite{devore2021ActaNumer,borjanzuazua2022Acta,HanJentzenE2018PNAS}. A plethora of research findings have consistently demonstrated that neural networks possess the capability to approximate smooth functions within certain classes with remarkable efficiency. The complexity of these approximations scales independently with respect to the dimensionality of the function, which is a critical factor influencing their performance. This independence on dimensionality stands in stark contrast to traditional numerical methods, which typically exhibit an exponential dependence on the dimensionality, typically referred to as  {\it curse of dimensionality}. As a result, neural networks offer a substantial advantage in terms of computational efficiency and scalability, especially when dealing with high-dimensional problems that are often encountered in practical applications.

Within the current work, our analysis is centered on the scenario where the approximating function is a shallow neural network, articulated as follows:
\begin{align}\label{eq:shallowNN}
    g_n(x) = \sum_{i=1}^{n} a_i \sigma(w_i \cdot x + b_i) + a_0.
\end{align}
Above, $w_i \in \mathbb{R}^d$ and $b_i \in \mathbb{R}$ represent the weights and biases of the hidden layer neurons, respectively. Similarly, $a_i \in \mathbb{R}$ and $a_0 \in \mathbb{R}$ serve as the weights and bias of the output layer. The function $\sigma(\cdot)$ in (\ref{eq:shallowNN}) is referred to as the activation function, which plays a pivotal role in introducing non-linearity into the network. This non-linearity is essential for the network to learn and approximate complex functions. Various activation functions are commonly used, including the Rectified Linear Unit (ReLU), Rectified Power Unit (RePU), sigmoid, tanh, and others.
To further conduct rigorous mathematical analysis on the shallow neural network (\ref{eq:shallowNN}), researchers have introduced a multitude of novel spaces, complementing the classical norms grounded in Hilbert space theory. These spaces are tailored to align with both the smoothness assumptions of the unknown function and the inherent structure of the neural network. Notable examples include the spectral Barron space \cite{MR1237720}, the (extended) Barron spaces \cite{EMaWu2022ConstrApp,LiLuMathePereverzev2024NM}, the variation spaces \cite{SiegelXu2023ConApp,SiegelXu2024FCom}, and the Radon-BV spaces  \cite{RarhiNowak2021JMLR}. For a comprehensive discussion encompassing these spaces, we refer to \cite{SiegelXu2023ConApp,LiLu2024arxiv} and the references cited therein.

Among these spaces and their corresponding function norms, the pioneering work presented in \cite{MR1237720} introduces the concept of the spectral Barron space. This space utilizes the Fourier coefficients of the unknown function.
Specifically, let $g$ denote the unknown function that we aim to approximate, and let $\hat{g}$ represent its Fourier transform. The spectral Barron space, as defined in \cite{MR1237720}, consists of functions whose weighted Fourier coefficients satisfy the boundedness condition:
\begin{align}\label{eq:weightFourierBarron}
    C_g := \int_{\mathbb{R}^{d}} |\hat{g}(\xi)| \cdot |\xi| \,d\xi <\infty.
\end{align}
Under this condition, there exists a shallow neural network $g_n$ of a certain form, often referred to as (\ref{eq:shallowNN}), that approximates the function $g$ on some bounded domain~$\Omega\subset \real^d$ with the following approximation rate:
\begin{align*}
    \|g_n-g\|_{L^2(\Omega)} \leq  C \sqrt{\frac{\abs{\Omega}}{n}},
\end{align*}
where~$\abs{\Omega}$ denotes the volume of the domain~$\Omega$;
and hence the convergence rate~$n^{-\frac{1}{2}}$ is independent of the spatial dimension $d$. From a mathematical perspective, it is intriguing to consider varying smoothness properties with respect to the Fourier coefficients $\hat{g}$ as implied by the plain definition of (\ref{eq:weightFourierBarron}) and to elucidate the intrinsic relationships among various functions that satisfy this condition.
This issue is addressed in \cite{KluBarron2018IEEE,Xu2020CiCP,SiegelXu2023ConApp, doi:10.1137/22M1478719}. 
In particular, the authors in~\cite{doi:10.1137/22M1478719} propose spaces,
with norm defined by 
\begin{align}\label{eq:spectralBarron_generalized}
    C^s_g := \int_{\mathbb{R}^{d}} |\hat{g}(\xi)| \cdot (1+|\xi|^2)^{\frac{s}{2}} \,d\xi <\infty.
\end{align}
Above, the index~$s\geq 0$ denotes the intrinsic smoothness.
Moreover, the authors construct a regularity theory within the spectral Barron space for the static Schr\"{o}dinger equation in the entire space~$\real^d$. They specifically consider non-negative indices~$s$ for the spectral Barron spaces, thereby laying a novel foundation for further exploration and understanding of the properties and behaviors of solutions to the static Schr\"{o}dinger equation in this unique functional space.
At the same time, there emerge numerous important new challenges. For instance, what is the role of negative indices $s$ in spectral Barron spaces? How do different spectral Barron spaces with various smoothness indices $s$ relate to each other? And how can we quantify the spectral Barron space in comparison to other function spaces? These new problems necessitate sophisticated analysis of the spectral Barron space as well as extensive further investigations.

It should be emphasized that the authors in~\cite{doi:10.1137/22M1478719} observe that the spaces with finite value~$C_g^s$ are closely related to some unbounded positive operator~$L$, which will be specified in detail later. This setting seems closely related to the {\it Hilbert scales} in the regularization theory of inverse problems. Let us briefly sketch the key concept of Hilbert scales, as e.g. outlined in~\cite{EHN96}. Denote $L$ to be a densely defined, injective, self-adjoint linear operator acting on a subset of a Hilbert space~$\mathcal{X}$, such that $L: D(L)\subset \mathcal{X} \rightarrow \mathcal{X}$ with a compact inverse. The Hilbert scales $\mathcal{X}_s$ are defined as the Hilbert spaces generated by the norm $\|x\|_s := \|L^s x\|_{\mathcal{X}}$. This allows for interpolation between different spaces from such Hilbert scales,  and to derive an error bound analysis for numerical algorithms aimed at solving inverse problems.

In the present work we will first introduce and explore a scale of spectral Barron spaces via a generating operator~$L=\ad$, with~$\Delta$  denoting the Laplacian on a suitable space. This construction then allows to understand this scale as a scale of interpolation spaces. We will highlight a corresponding  {\it interpolation inequality} to underscore the scale property of these spectral Barron spaces. 
\par
Then we
turn to inverse problems~$y=F(x)$, aiming to reconstruct~$x$ from noisy data around~$y$. If the governing mapping~$F$ fits the scale of spectral Barron spaces, and this will be formulated via a {\it link condition}, then a {\it conditional stability} estimate holds true. We will provide three illustrative examples where such a relationship holds true. 
\par
Lastly, we will explore the reconstruction of the solution to inverse problems. By a
  specific Tikhonov regularization method, wherein the penalty term is
  specifically tailored to incorporate Barron scales. As a main result we shall
  obtain an error bound in terms of the noise level. In addition we
  shall provide a result on the complexity for achieving this rate in
  terms of the width of a shallow neural network.
Although this approach is academic, it highlights that the curse of dimensionality may be overcome when solving ill-posed problems tailored to a scale of spectral Barron spaces.
To this end we will recover the well known~$1/\sqrt n$ approximation
rate by shallow neural networks, as first established in the
seminal study~\cite{MR1237720}. We shall do this from a perspective of
integral operators with values in a Hilbert space.

\section{Spectral Barron spaces and their interpolation}
\label{sec:interpolation}
In this section, we focus on the interpolation property of spectral Barron spaces on~$\mathbb{R}^{d}$. To accomplish this, it is necessary to formally define both the spectral Barron spaces and scales of these spaces.

\subsection{Spectral Barron spaces on ~$\real^{d}$ and their scales}
Our construction will start at the following spectral Barron space defined by
\begin{equation}
  \label{eq:B0}
  \Bsp 0:= \set{f \in \Sp,\quad \int_{\real^{d}}\abs{\hat f(\xi)}\; d\xi<\infty},
\end{equation}
which consists of all functions from~$\Sp$, the tempered distributions, i.e., the dual of the Schwartz space of rapidly decaying functions on~$\real^d$, having an absolutely integrable Fourier coefficients, see~\cite{MR248435} for details. 

To further investigate scales of spectral Barron spaces, we introduce the unbounded positive Laplace operator~$-\Delta$, and consider the weighted (unbounded) operator
\begin{equation*}
	L := \ad \colon \Bsp 0 \to \Bsp 0. 
\end{equation*}
Subsequently, the negative powers of~$L$ will be shown to be bounded operators within~$\Bsp 0$. To proceed with our analysis, it is imperative to establish additional properties of the operator~$\oned$. The pertinent concept in this context is captured, for instance, in~\cite{MR672408} (where it is referred to as weakly positive), or in the comprehensive treatment provided by the monograph~\cite{MR2244037}.
\begin{definition} [Weakly sectorial operator] An (unbounded) operator $A \colon \domain(A) \subset X \to X$, acting on a Banach space $X$, is called weakly sectorial if, for all $t > 0$, the inverse $(tI + A)^{-1}$ exists, and there exists a constant $M$ such that the norm bound $\norm{(tI + A)^{-1}}{X \to X} \leq \frac{M}{t}$ holds for all $t > 0$.
\end{definition}
These operators are important because they enable a functional calculus, and we will subsequently utilize some of their properties in our discussion.
\begin{prop}
  For every~$n\in\nat$ the operator~$\oned^{n}$ is weakly sectorial on~$\Bsp 0$.
\end{prop}
\begin{proof}
  We need to verify that~$\norm{\lr{ t + \oned^{n}}^{-1}}{\Bsp 0 \to \Bsp 0}\leq
  M/t$ for some finite constant~$M$.
  
  First, by the definition of the
  Laplacian, the action of~$\oned$ on the Fourier domain is a multiplication
  by~$1 + \abs{\xi}^{2},\ \xi\in\real^{d}$. We thus introduce the following
  abbreviation~$\left< \xi\right> := (1 + \abs{\xi}^{2})^{1/2}$.
  By iterating the
  above action~$n$-times we find that the action of~$\oned^{n}$ on the Fourier domain is a multiplication
  by~$\left< \xi\right>^{2n},\ \xi\in\real^{d}$.

Now, let~$f$ be any
  function in~$\Bsp 0$, by taking Fourier transform we then observe that
  \begin{equation}
    \label{eq:fourier-n}
\F\lr{\lr{t I + \oned^{n}}^{-1} f}(\xi) = \frac{\hat f(\xi)}{ t + \left<\xi\right>^{2n}},\quad \xi\in\real^{d}.
  \end{equation}
By taking the ~$\Bsp 0$ norm we find
$$
\norm{\lr{ t I + \oned^{n}}^{-1} f}{\Bsp 0} =
\int_{\real^{d}}\frac{\abs{\hat f(\xi)}}{t + \left<\xi \right>^{2n}}\; d\xi\leq \frac 1 t \norm{f}{\Bsp 0},\ t>0,
$$
which completes the proof.
\end{proof}

We are to extend this to fractional powers~$\oned^{s}$ for~$0 < s < n$,
given from its real-valued representation
\begin{equation}\label{eq:real-power}
  \tau^{\alpha} = \frac{\sin(\pi\alpha)}{\pi}
  \int_{0}^{\infty} s^{\alpha-1}\lr{s + \tau}^{-1}\tau \; ds, \quad \tau >0,
\end{equation}
  as follows. This is often called the {\it Balakrishnan Representation}, see e.g.~\cite{MR2244037}.
\begin{definition}\label{de:frac-power}
  Let~$0 < \alpha < 1$. For a weakly sectorial operator~$A\colon X \to X$ its fractional power~$A^{\alpha}$ is given from the integral representation
  $$
  A^{\alpha} u = \frac{\sin(\pi\alpha)}{\pi}
  \int_{0}^{\infty} s^{\alpha-1}\lr{s I + A}^{-1}A u\; ds,\quad
  u\in X.
  $$
\end{definition}

Turning to Fourier domain we then have that
$$
\mathcal F\lr{A^{\alpha} u} = \frac{\sin(\pi\alpha)}{\pi}
\int_{0}^{\infty} s^{\alpha-1}\mathcal F\lr{\lr{s I + A}^{-1}A u}\; ds. 
$$
Analogous to~(\ref{eq:fourier-n}) we find with~$A:= \oned^{n}$, that
$$
\F\lr{\lr{s I + A}^{-1}A u}(\xi) = \frac{\left<\xi\right>^{2n} \hat u(\xi)}{s + \left<\xi\right>^{2n}},\quad \xi \in\real^{d}.
$$
Now, for the positive real number~$\tau:= \left< \xi\right>^{2n}$ the real integral~\eqref{eq:real-power} yields
$$
\left<\xi\right>^{2\alpha n} = \frac{\sin(\pi\alpha)}{\pi}
 \int_{0}^{\infty} s^{\alpha-1}\lr{s + \left< \xi\right>^{2n} }^{-1} \left< \xi\right>^{2n} \; ds,
$$
and hence that
$$
\F\lr{\oned^{\alpha n} u}(\xi) = \left<\xi\right>^{2\alpha
  n}\hat u(\xi).
$$
Specifically, for~$\alpha:= s/(2n)< 1$ this gives
\begin{equation}
  \label{eq:fourier-s}
\mathcal F\lr{\oned^{s/2} u}(\xi) = \left<\xi\right>^{s}\hat u(\xi).
\end{equation}
This identity will be used in the subsequent discussion.



Now we may use the above formulation to generate a scale of spectral Barron spaces. For any fixed~$s\geq 0$ we assign the set
\begin{equation}
  \label{eq:Bs-set}
  \Bsp s := \range\lr{\oned^{-s/2}},
\end{equation}
as the range of the operator~$\oned^{-s/2}$ acting on~$\Bsp 0$.

It was shown in~\cite[Prop.~3.8]{doi:10.1137/22M1478719} that the operator~$\oned^{-1}\colon \Bsp{s} \to \Bsp{s}$ is injective whenever~$s \geq 0$. From this, we can readily deduce that this property holds for arbitrary non-negative powers~$\oned^{-s/2}\colon \Bsp{0} \to \Bsp{0}$. Hence, for any~$g \in \Bsp{s}$, there exists a unique source element~$h \in \Bsp{0}$ such that~$g = \oned^{-s/2} h$. We thus define the norm as follows:
$$
\norm{g}{\Bsp{s}} = \norm{h}{\Bsp{0}},\quad \text{for all}\ g \in \Bsp{s}.
$$
In view of~(\ref{eq:fourier-s}), we have the following implications and properties.

\begin{prop}\label{prop:bsp-norm}
  For~$s\geq 0$ and~$g\in\Bsp s$, the norm in $\Bsp s$ is defined by
$$
\norm{g}{\Bsp s} = \int_{\real^{d}}\abs{\hat g} \left<\xi\right>^{s}\; d\xi.
$$
\end{prop}
Hence the spaces~$\Bsp s$, equipped with the norm~$\norm{\cdot}{\Bsp
  s} $, coincide with the original spectral Barron spaces, 
when considered on~$\real^{d}$, see the
reference~\cite{doi:10.1137/22M1478719}. 
This is consistent with the original definition of the spectral Barron norm up to a constant factor, and for simplicity's sake, we refer to $\Bsp s$ as a spectral Barron space, too. We emphasize that
a plain form of the spectral Barron spaces~$\Bsp s$ were already introduced in the
monograph~\cite[Chapt.~2]{MR248435}. 

It is evident from this construction that spaces with higher $s>0$ values are smaller, given that $\oned^{-1}$ is a bounded operator. Consequently, we can regard balls in $\Bsp s$ as smoothness classes (or source sets), as these are typically utilized in regularization theory of inverse problems, we refer to the monograph~\cite{EHN96}, again.
\par
Also, the above definition from Proposition~\ref{prop:bsp-norm} extends to negative values~$s\in\real$. In this case, the mapping~$\oned^{-s/2}$ is unbounded (on~$\Bsp 0$), and the corresponding spectral Barron spaces
are obtained as~$\Bsp s := \range\lr{\oned^{-s/2}} \subset U$, where $U \supset \Bsp 0$ is some universal superspace, attached to the governing operator~$\oned$. 
We refer to the detailed construction provided in~\cite[Section~6.3]{MR2183483}. The formula for evaluating the corresponding norms expands upon the one presented in Proposition~\ref{prop:bsp-norm}.

We conclude this section with a concise discussion on the results presented in~\cite[Lem.~3.2]{doi:10.1137/22M1478719}. Specifically, the initial statement in that lemma can be reinterpreted to mean that the operator~$-\Delta\colon \Bsp s\to \Bsp s$ possesses weak sectoriality.

The second assertion in the statement of \cite[Lem.~3.2]{doi:10.1137/22M1478719} claims that the norm of the operator $\oned^{-1}$ from $\Bsp{s}$ to $\Bsp{s+2}$ is bounded by one. Given that we have identified the spectral Barron spaces as the ranges in~(\ref{eq:Bs-set}) through Proposition~\ref{prop:bsp-norm}, we proceed with the following argument. Let $g \in \Bsp{s}$ with a corresponding source element $h \in \Bsp{0}$. Then, it is straightforward to observe that
$$
\oned^{-1}g = \oned^{-1} \oned^{-s/2}h = \oned^{-(s+2)/2}h,
$$
which implies that
$$
\norm{\oned^{-1}g}{\Bsp{s+2}} = \norm{\oned^{-(s+2)/2}h}{\Bsp{s+2}} = \norm{h}{\Bsp{0}} = \norm{g}{\Bsp{s}}.
$$
Thus, we confirm by such arguments that the norm of $\oned^{-1}$ from $\Bsp{s}$ to $\Bsp{s+2}$ is indeed bounded by one.

\subsection{Interpolation between spectral Barron spaces}
\label{sec:smoothness}

As mentioned previously, weakly sectorial operators possess a functional calculus, as exhaustively detailed in~\cite{MR2244037}. In \cite[Sec.~6]{MR2244037}, the theory of interpolation spaces is outlined. It is confirmed therein that the scale of spaces induced by a weakly sectorial operator, specifically the spectral Barron spaces in our context, constitutes a scale of {\it real interpolation spaces}.

A well-known result, presented in \cite{MR2244037}, is the following {\it moment inequality}, which holds for some constant $C$ whenever $0 < \alpha < 1$:
\begin{equation}
  \label{eq:moment}
  \norm{A^{\alpha} x}{X} \leq C \norm{A x}{X}^{\alpha} \norm{x}{X}^{1-\alpha},\quad x \in X.
\end{equation}
Given any triple $0 \leq  r < s < t$, we apply this moment inequality (\ref{eq:moment}) by setting $\alpha := \frac{s-r}{t-r}$, $x := \oned^{r/2}u$, and $A := \oned^{(t-r)/2}$. 
This yields the following result.
\begin{prop}
Given any triple $0 \leq r < s < t <\infty$ we have for~$u \in \Bsp{t}$ the inequality
\begin{equation}
  \label{eq:interpolation-first}
  \norm{\oned^{s/2}u}{\Bsp{0}} \leq C \norm{\oned^{r/2}u}{\Bsp{0}}^{\frac{t-s}{t-r}} \norm{\oned^{t/2}u}{\Bsp{0}}^{\frac{s-r}{t-r}}.
\end{equation}
\end{prop}
The above inequality~(\ref{eq:interpolation-first}) can be recast as follows.
\begin{theorem}[interpolation inequality]
    For any triple $r < s < t <\infty$, the following inequality holds for all ~$u \in \Bsp{t}$ such that
\begin{equation}
  \label{eq:interpolation}
  \norm{u}{\Bsp{s}} \leq C \norm{u}{\Bsp{r}}^{\theta} \norm{u}{\Bsp{t}}^{1 -\theta},\quad u \in \Bsp{t},
\end{equation}
where $0 < \theta := \frac{t-s}{t-r} < 1$.
\end{theorem}
It is important to note that for this inequality to hold, we must ensure that $u$ belongs to the smallest space, corresponding to the largest value~$t$.
\begin{remark}
    A similar inequality as~\eqref{eq:interpolation} was obtained previously in the study~\cite{liao2023spectralbarronspacedeep} by using Hölder-type arguments.
\end{remark}

 Fixing a pair of parameters~$a,p>0$, and applying the above moment inequality~\eqref{eq:interpolation} with~$u:= \oned^{-(r+a)}v$,~$t:= p+ (r+a)$ then we find the generalized form
 \begin{equation}
 \label{eq:moment-general}
      \norm{v}{\Bsp{0}} \leq C \norm{v}{\Bsp{-a}}^{\frac p {p+a}} \norm{v}{\Bsp{p}}^{\frac a {p+a}},\quad v \in \Bsp{p},
 \end{equation}
 where the parameter~$\theta$ in~\eqref{eq:interpolation} evaluates as~$\theta:= \frac p {p+a}$. 
 
\section{Inverse problems in spectral Barron spaces}
\label{sec:ill-posed}
We shall use the aforementioned constructions to investigate inverse problems within the context of spectral Barron spaces. Our methodology draws inspiration from the corresponding theory of inverse problems in Hilbert scales, as outlined in the seminal study~\cite{MR762862}, see also the monograph~\cite{MR248435}, with the aim of addressing approximations via neural networks.
Suppose we are presented with a (non-linear) mapping:
\begin{equation}
  \label{eq:F-def}
F\colon \domain(F) \longrightarrow Y,
\end{equation}
where $Y$ is a Banach space, and our objective is to (approximately)
solve the equation $y = F(u)$ for $u \in \domain(F)$. However, rather
than having access to the exact data~$y \in Y$, we are only provided with noisy data of the form:
\begin{equation}
  \label{eq:noisy}
  \norm{\yd - y}{Y} \leq \delta,
\end{equation}
where $\delta > 0$ represents the noise level. Let $\udag$ denote the unknown solution, implying that we consider data $\yd = F(\udag) + \xi$ with $\norm{\xi}{Y} \leq \delta$.
\par
We aim to utilize the \(1/\sqrt{n}\) approximation rate achievable by shallow neural networks in spectral Barron spaces, as established in Corollary~\ref{cor:NN-approx}. Notably, this approximation rate holds only on bounded domains. Consequently, in this work we restrict our attention to inverse problems with a forward operator \(F\) whose domain \(\domain(F)\) consists of elements in spectral Barron spaces with support contained in a given compact set \(K \subset \real^d\).

Following the framework outlined in the monograph~\cite{MR248435}, we consider the spaces \(\Bsp{s} \cap \mathcal{E}'(K)\), where \(\mathcal{E}'(K)\) denotes the set of distributions supported within the compact set \(K\).
\begin{remark}\label{rem:paley}
We emphasize that this function space is rich in structure. By the Paley-Wiener theorem (see, e.g., \cite[Thm.~1.7.7]{MR248435}), the Fourier transform of any compactly supported smooth function~$u \in C^\infty_0(\mathbb{R}^d)$ with support contained in~$K$ can be characterized as follows: for every~$N > 0$, there exists a constant~$C_N$ such that
$$
\abs{\hat u(\xi)} \leq C_N < \xi>^{-N},\quad \xi\in\real^d.
$$
Consequently, the norm of~$u$ can be bounded as
$$
\norm{u}{\Bsp s} \leq C_N\int_{\real^ d}<\xi>^s<\xi>^{-N}\; d\xi,
$$
which is finite for~$N$ large enough. Consequently, these functions lie in every Barron space.
\end{remark}
Thus we will assume that the unknown solution~$\udag$ lies in the set~$\mathcal{M}(R)$, which is defined as follows:
\begin{equation}
  \label{eq:MR}
  \mathcal{M}(R) := \{z \in \domain(F) \cap \Bsp{p}\cap \mathcal E^\prime(K) \mid \norm{z - \uast}{\Bsp{p}} \leq R\},
\end{equation}
for a fixed reference element \( \uast \in \Bspo{p} \), and encapsulates our prior knowledge about the unknown solution \( \udag \).
From a practical perspective, the restriction to functions with compact support is often a natural assumption.  
This is particularly evident in the case of the Radon transform, where the solution represents the density of a bounded physical object.
\par
Now we face two tasks.  
First, we need to align the forward mapping \( F \) with the scale of Barron spaces.  
Additionally, we need to design an algorithmic scheme that yields neural networks which achieve a certain approximation rate as the noise level \( \delta \to 0 \), and which have a prescribed width \( n = n(\delta) \).

\subsection{Aligning the inverse problem to Barron spaces}
\label{sec:align-inverse-probl}

In order to address such inverse problems within the framework of
spectral Barron spaces, we require that the mapping $F$ aligns with
the properties and characteristics of the following abstract scale
assumption.
The mapping~$F$ is associated with the scale~$\Bsp s$ in the following manner.
\begin{ass}
  [Link condition]\label{ass:link}
  For a given parameter $a > 0$, there exist constants~$0 < m\leq M <\infty$ such that, for any $u_1, u_2 \in \domain(F) \cap \Bsp{-a}$, the following inequalities hold:
$$
m \norm{u_1-u_2}{\Bsp{-a}} \leq \norm{F(u_1) - F(u_2)}{Y}\leq M \norm{u_1-u_2}{\Bsp{-a}}.
$$
\end{ass}
This link condition quantifies the degree of smoothing properties exhibited by the mapping $F$. In particular it guarantees that the mapping~$F\colon \domain(F)\cap \Bsp {-a}\to Y$ is norm-continuous.

The main result here is as follows.
\begin{theorem}[Conditional stability]\label{thm:cond-stable} 
  Under Assumption~\ref{ass:link}, there exists a constant $C < \infty$ such that:
  \[
  \norm{u_1 - u_2}{L^{2}(\Omega)} \leq C \norm{F(u_1) -
    F(u_2)}{Y}^{\frac{p}{p+a}} R^{\frac a {p+a}}, \quad \text{for all } u_1, u_2 \in \mathcal{M}(R).
  \]    
\end{theorem}
\begin{proof}
We start from the interpolation inequality \eqref{eq:moment-general}. Substituting the left-hand sides of the linking inequalities in Assumption~\ref{ass:link} into \eqref{eq:moment-general}, and noting that the relevant elements belong to the set $\mathcal M(R)$ defined in \eqref{eq:MR}, we obtain the desired result.
\end{proof}

Interpolation may also be employed to bound the {\it modulus of continuity} of the (inverse of) the mapping $F$ at the unknown solution~$\udag\in\mathcal M(R)$. 
\begin{definition}[Modulus of continuity]\label{def:modulus} 
Let $\delta>0$. 
For a mapping $F\colon \mathcal{M}(R) \subset \mathcal{D}(F) \to Y$, and at a point $\udag \in \mathcal{M}(R)$, we define the modulus of continuity of (the inverse of) $F$ as follows:
\begin{equation*}
\omega_{\udag}(F, \mathcal{M}(R), \delta) := \sup \left\{ \norm{\udag - u}{L_2(\Omega)} \,|\, u \in \mathcal{M}(R), \, \norm{F(\udag) - F(u)}{Y} \leq \delta \right\}.
\end{equation*}
\end{definition}
This yields the immediate corollary to
Theorem~\ref{thm:cond-stable} .
\begin{corollary}\label{cor:modulus}
  Suppose that the unknown solution~$\udag \in\mathcal M(R)$. Under
  Assumption~\ref{ass:link} there holds
  $$
 \omega(F,\mathcal M(R),\delta) \leq C R^{\frac{a}{p+a}}
 \delta^{\frac{p}{p+a}},\quad \delta>0.
  $$
\end{corollary}
The modulus of continuity plays a crucial role for establishing lower
bounds for approximation of the unknown solution in inverse
problems. Specifically, if the mapping~$F$ were linear, then the
set~$\mathcal M(R)$ is centrally symmetric, i.e.,\ with~$u\in\mathcal
M(R)$ we also have that~$-u\in\mathcal M(R)$. In this case any
reconstruction, say~$\ud := \mathcal{R}(\yd)$ with any inversion algorithm $\mathcal{R}$, based on noisy data~$\yd$
from~(\ref{eq:noisy}) will have an error estimate bounded from below by the
modulus of continuity. This means that
$$
\sup_{u\in\mathcal M(R),\ \norm{F(u) - \yd}{Y}\leq
  \delta}\norm{\udag - \mathcal{R}(\yd)}{L_{2}(\Omega)} \geq \omega(F,\mathcal M(R),\delta).
$$
A proof is given in~\cite[Thm.~3.1]{MR859375} in case of Hilbert spaces, but the arguments are valid in arbitrary Banach spaces. 
\begin{remark}
We emphasize that, thus far, we have only established an upper bound in Corollary~\ref{cor:modulus}. To discuss the optimality of certain reconstructions, lower bounds are necessary. In the context of inverse problems in Hilbert scales, such lower bounds have been provided in~\cite{MR1984890}.
\end{remark}

\subsection{Verifying a link condition}
\label{sec:link-verify}
In the above discussion, we presented an abstract Assumption \ref{ass:link} that plays a crucial role in establishing the conditional stability, as outlined in Theorem \ref{thm:cond-stable}.

Here, we will present three linear scenarios where the link condition
specified in Assumption~\ref{ass:link} can be verified. In the first
case, we utilize tools from the theory of pseudo-differential
operators to demonstrate that for such mappings,  existence of the
link condition is a natural consequence. In the second case, we rely
on recent research on the Schr\"odinger equation in spectral Barron
spaces, and we find that the link condition holds true for the
associated inverse problem. Finally, the last case focuses on the
classic Radon transform, providing further insights into the realm of
spectral Barron spaces.

\subsubsection{Perspective from pseudo-differential operators}
\label{sec:perspective}

Firstly, we consider {\it pseudo-differential operators} (\pdos), which are defined through symbols, specifically smooth mappings denoted as $\phi(x,\xi)$, where $x,\xi \in \mathbb{R}^{d}$, via the equation:
\begin{equation}
  \label{eq:pdo}
 F(u)(x):= (T_{\phi}u)(x) := \mathcal{F}^{-1}[\phi(x,\xi) \mathcal{F}u](\xi)\Big|_{\xi=x},\quad x\in\mathbb{R}^{d}.
\end{equation}
Above, $\mathcal{F}$ denotes the Fourier transform. These operators are well-known to map from $\mathcal{S}'(\real^d)$ to $\mathcal{S}'(\real^d)$. Numerous texts discuss \pdos, and for the present basic considerations, we refer to the monographs~\cite{MR248435} and~\cite{MR2743652}.
\begin{xmpl}
  [Powers of Laplacian]\label{xmpl:laplacian}
  Consider the specific symbol
  \begin{equation}
    \label{eq:laplace} \phi_{1}(x,\xi) := \left< \xi \right>,\quad \xi \in \real^{d}.
\end{equation}
Then we see that~$T_{1}:= T_{\phi_{1}}= (I - \Delta)^{1/2}$.
 More general, for any parameter~$s\in\real$ we may consider
 accordingly~$T_{s} = (I - \Delta)^{s/2}$ with symbol~$\left<\xi \right>^{s}$. In
 this case, for~$u\in\Sp$, we find that
 $$
 \norm{T_{s}u}{\Bsp 0} = \norm{\F^{-1} \left< \xi \right>^{s} \F u}{\Bsp 0} =
 \int_{\real^{d}}\left<\xi\right>^{s}\abs{\hat u(\xi)}\; d\xi,
 $$
and hence~$\norm{T_{s}u}{\Bsp 0} = \norm{u}{\Bsp s},\ u\in \Bsp s$,
see Proposition~\ref{prop:bsp-norm}.
\end{xmpl}
This approach allows us to exhibit that there are many \pdos~(linear
mappings~$F$) for which the link condition holds.
\begin{prop}\label{prop:-a}
  Suppose that~$F$ is a \pdo{} with symbol~$\phi(x,\xi)$ satisfying the inequality
  \begin{equation}
    \label{eq:symbol-bound}
    0 < c \leq \phi(x,\xi) \left<\xi\right>^{a} \leq C < \infty.
  \end{equation}
  Then~$F\colon \Bsp {-a} \to Y:= \Bsp 0$ obeys the link condition
  from Assumption~\ref{ass:link}. Moreover, if~$a >
  d$ then~$F$ is a trace-class operator with square integrable
  (Mercer) kernel.
\end{prop}
\begin{proof}
Rewriting the condition in Assumption~\ref{ass:link}, we can argue as follows. The link condition certainly holds if the composition $F
  \circ L^{a}\colon \Bsp 0\to \Bsp 0$  is an isomorphism. Since the symbol $\langle \xi \rangle^{a}$ of the operator $L^{a/2}$ does not depend on the spatial variable $x$, the composition $F \circ L^{a}$ has the symbol $\phi(x,\xi)\langle \xi \rangle^{a}$. By assumption, this symbol is bounded away from zero and infinity. Therefore, we can bound the norm from above and below as follows:
  \begin{align*}
    \norm{(F\circ L^{a}) u}{\Bsp 0} &= \norm{\F^{-1} \phi(x,\xi) \left< \xi\right>^{a} \F u}{\Bsp 0}\\
    &= \int_{\real^{d}} \phi(x,\xi) \left< \xi\right>^{a} \abs{\hat u(\xi)}\; d\xi\\
    & \asymp \norm{u}{\Bsp 0},\quad u\in\Bsp 0, 
  \end{align*}
where $\mathcal{F} u(\xi)$ denotes the Fourier transform of $u$. Finally, if $a > d$, then the operator $F$ is of the trace class, which implies that a Mercer kernel exists.
\end{proof}

\begin{remark}
Symbols that obey condition~(\ref{eq:symbol-bound}) are typically referred to as $(-a/2)$-elliptic. In other words, the proposition establishes that $(-a/2)$-elliptic operators map elements from $\Bsp {-a}$ to $\Bsp 0$.
\end{remark}

\subsubsection{Inverse problem in the Schr{\"o}dinger equation}
\label{sec:schroedinger}

The recent study~\cite{doi:10.1137/22M1478719} highlighted the action
of the operator governing the Schrödinger equation
\begin{equation}
  \label{eq:schroedinger}
- \Delta y + V y =u \quad \text{in}\quad \real^{d},
\end{equation}
within the class of spectral Barron spaces~$\Bsp s$, for~$s\geq 0$. One
highlight result \cite[Thm 2.3]{doi:10.1137/22M1478719} there was to establish a conditional
stability estimate
\begin{equation}
  \label{eq:stability-schroedinger}
  \norm{y}{\Bsp{s+2} } \leq C \norm{u}{\Bsp s},
\end{equation}
for~$s\geq 0$ whenever the pair~$(y,u)$ obeys~(\ref{eq:schroedinger})
for a potential of the form~$V:= \alpha + W\geq 0$, with~$\alpha>0$ and~$W
\in\Bsp s$. These authors use the theory of Fredholm operators, and
hence the governing operator for~(\ref{eq:schroedinger}) is given
as
\begin{equation}
  \label{eq:governing-op}
  \lr{\alpha - \Delta}\lr{I + \taw} y = u,
\end{equation}
with~$\taw$ as in \cite[Eq. (3.2)]{doi:10.1137/22M1478719} as
\begin{equation*}
  \taw y := \lr{\alpha - \Delta}^{-1} \lr{W y},\quad y\in \Bsp s.
\end{equation*}
The crucial point in the analysis is to first establish that the
operator~$\taw\colon \Bsp s\to \Bsp s$ is bounded and compact.  Secondly, since in~\cite[Prop.~3.8]{doi:10.1137/22M1478719} the operator~$I + \taw$ is shown to be injective, it has a bounded inverse in~$\Bsp s$ by the Fredholm alternative.
From this the (even two
sided) stability estimate~(\ref{eq:stability-schroedinger}) is a
simple consequence. The above analysis was carried out for~$s\geq 0$.

Recall that the direct problem involves deriving $y$ given a source $u$. Conversely, in the inverse problem we aim to recover the source $u$ given $y$. Therefore, the forward mapping, which is actually the inverse of the process we are interested in, is given by:
\begin{equation}
  \label{eq:F-inverse}
F:= \lr{I + \taw}^{-1}\lr{\alpha - \Delta}^{-1}\colon \Bsp {-2} \to \Bsp 0,
\end{equation}
where this mapping is smoothing. We want to verify for the present case the link
condition from Assumption~\ref{ass:link}, and we hence need this to
hold for~$a:= -2$. Since we already know that~$\lr{I + \taw}$ is
bounded, with bounded inverse,  in~$\Bsp 0$, it is enough to validate
that the mapping
\begin{equation}
  \label{eq:a-d}
\lr{\alpha
  - \Delta}^{-1}\colon \Bsp {-2}\to \Bsp 0,
\end{equation} is $(-1)$-elliptic, because in this case
Proposition~\ref{prop:-a} applies. But the
condition~(\ref{eq:symbol-bound}) for~$a:= 2$ is easily checked for the mapping~$\lr{\alpha - \Delta}^{-1}$ with symbol~$1/(\alpha +
\abs{\xi}^{2})$.

\subsubsection{Radon transform in spectral Barron spaces}
\label{sec:radon}

Here, we outline how the Radon transform can be considered as operating in spectral Barron spaces. The Radon transform on~$\mathbb{R}^{d}$ is defined as:
\begin{equation}
  \label{eq:radon-def}
     (R u)(\zeta,\kappa) := \int_{x: \langle x, \kappa \rangle = \zeta} u(x)\; dx,\quad (\zeta,\kappa)\in Z:= \mathbb{R} \times S^{d-1},
   \end{equation}
where $\langle x, \kappa \rangle$ denotes the dot product between $x$ and $\kappa$. Extensive literature exists on the analysis of the Radon transform in various spaces, with the monograph~\cite{MR856916} providing a comprehensive overview. Our discussion builds on the recent study~\cite{kindermann_2025}, which examines spaces parameterized by $s$ (smoothness), $p$ (integrability), and $t$ (for fine-tuning). Notably, the spectral Barron spaces~$B^s({\real^{d}})$ correspond to $S^{s,1}_{0}(\mathbb{R}^{d})$ for $s \in \mathbb{R}$ in~\cite{kindermann_2025}.
From~(\ref{eq:radon-def}), it is evident that the range of the Radon transform exhibits a specific structure known as a {\it sinogram}. Consequently, the authors in~\cite{kindermann_2025} introduce {\it  sinogram spaces}~$\mathfrak{R}^{s,p}_{t}(Z)$ with parameters having similar interpretations. The fundamental identity, as specified in our context and detailed in~\cite[Thm.~2]{kindermann_2025}, is given by:
\begin{equation}
  \label{eq:identity-Radon}
     \|u\|_{B^s({\real^{d}})} = \|R u\|_{\mathfrak{R}^{s + d-1,1}_{d-1}(Z)},\quad u\in B^s({\real^{d}}),
   \end{equation}
where, for a function $y \in \mathfrak{R}^{s + d-1,1}_{d-1}(Z)$, its sinogram norm is defined as:
\begin{equation}
  \label{eq:R-norm}
     \|y\|_{\mathfrak{R}^{s + d-1,1}_{d-1}(Z)} = \int_{S^{d-1}}\|\mathcal{F}_{1} y(\cdot,\kappa)\|_{S^{s+d-1,1}_{d-1}(\mathbb{R})}\; d\kappa.
   \end{equation}
Here, $\mathcal{F}_{1}$ denotes the Fourier transform applied solely to the offset parameter~$\zeta$.

In view of other estimates, specifically we refer
to~\cite[Thm.~3]{kindermann_2025}, it is reasonable to assume that the Radon transform on~$\mathbb{R}^{d}$ exhibits smoothing properties with order~$d-1$. We denote this smoothing effect by setting~$s := -(d-1)$. Then, the identity in~(\ref{eq:identity-Radon}) specializes to:
\begin{equation}
  \label{eq:radon-barron-link}
  \|u\|_{B^{-(d-1)}(\real^{d})} = \|R u\|_{\mathfrak{R}^{0,1}_{d-1}(Z)}.
\end{equation}
The dependence on the parameter~$t := d-1$ in the norm on the right-hand side is minimal, as demonstrated by:
\begin{align*}
  \|R u\|_{\mathfrak{R}^{0,1}_{d-1}(Z)} &= \int_{S^{d-1}}\int_{\mathbb{R}}|\mathcal{F}_{1}R u(\cdot,\kappa)(\xi)|\left(\frac{|\xi|}{\langle\xi\rangle}\right)^{d-1} \, d\xi \, d\kappa \\
  &\leq \int_{S^{d-1}}\int_{\mathbb{R}}|\mathcal{F}_{1}R u(\cdot,\kappa)(\xi)| \, d\xi \, d\kappa \\
  &= \int_{S^{d-1}} \norm{R u(\cdot,\kappa)}{B^0(\real)}   \, d\kappa \\
  &= \|R u\|_{\mathfrak{R}^{0,1}_{0}(Z)}.
\end{align*}
We recognize that we can consider sinograms~$y = R u \in \text{range}(R)$ with the property that
$$
\mathbb{E}_\kappa\norm{R u(\cdot,\kappa)}{B^0(\real)}
$$ is finite. Here, the quantity~$\mathbb{E}_\kappa$ represents the expectation with respect to the Haar measure on~$S^{d-1}$ for the angle~$\kappa$. The related question of which sinogram spaces to assume for noisy data when solving inverse problems necessitates further investigation.
   \subsection{Universal Approximation in spectral Barron spaces}
\label{sec:appr-spectr-barr}
Below, we shall utilize Barron's fundamental result from his seminal work~\cite{MR1237720} concerning approximation by neural networks within the framework of spectral Barron spaces. Our objective here is not to delve into the latest advancements in refining approximation rates. Rather, we aim to present a distinct perspective on this problem.

The fundamental approximation rates are derived using probabilistic
arguments, relying on the fact that any given function can be
represented as an expectation with respect to some probability
measure. Additionally, the identification of an appropriate Hilbert
space plays a crucial role. This combination allows us to conclude
that the Mean Integrated Squared Error (MISE) is bounded by a
parametric rate of~$1/n$. In principle, such result was shown
in~\cite{Xu2020CiCP}. Here we will demonstrate this result by employing a specific integral operator.

Precisely, given a specific activation function~$\sigma$, the output of any neuron, for an input~$x \in \mathbb{R}^{d}$, can be expressed as~$\beta \sigma(\langle x, \omega \rangle + b)$, where~$\omega \in \mathbb{R}^{d}$ represents the weights, $b \in \mathbb{R}$ is the bias, and~$\beta \in \mathbb{R}$ is another weight factor in the outer layer. We consider a set~$G$ of parameters~$\theta := (\omega, b, \beta)$. This leads us to define the kernel:
\begin{equation}
  \label{eq:kernel}
  (x, \theta) \in \mathbb{R}^{d} \times G \mapsto g(x, \theta) := \beta \sigma(\langle x, \omega \rangle + b).
\end{equation}
This kernel gives rise to the following {\it integral operator}:
\begin{equation}
  \label{eq:integral-operator}
  (T_{g}\rho)(x) := \int_{G} g(x, \theta) \rho(\theta) \, d\theta, \quad x \in \mathbb{R}^{d}.
\end{equation}
For our analysis to be valid, we need to establish conditions that ensure the operator maps into~$L_{2}(\Omega)$ for some bounded domain~$\Omega \subset \mathbb{R}^{d}$. We will employ assumptions regarding the decay rate of~$\rho$. Specifically, we assume that for some parameter~$s > 0$, the function~$\rho$ belongs to the Banach space:
\begin{equation}
  \label{eq:rho-Banach}
  H^{s}_{1}(G) := \left\{ \rho : \int_{G} \|\theta\|^{s} |\rho(\theta)| \, d\theta < \infty \right\},
\end{equation}
where~$\|\theta\| = \norm{\omega}{2} + \abs{b}+\abs{\beta}$.
We will later tailor these assumptions to the specific problem at hand. Here, we claim the following proposition:
\begin{prop}\label{prop:integral-L2}
  Suppose that~$\Omega$ is a bounded domain and that the kernel~$g$ satisfies a growth condition given by:
  \begin{equation}
    \label{eq:kernel-growth}
    \sup_{x \in \Omega} |g(x, \theta)| \leq C_s \|\theta\|^{s}, \quad \theta \in G.
  \end{equation}
  Then the operator~$T_{g}$ defined in~(\ref{eq:integral-operator}) maps~$H^{s}_{1}(G)$ into~$L_{2}(\Omega)$.
\end{prop}

\begin{proof}
For any fixed $\rho \in H^{s}_{1}(G)$, direct calculation reveals
\begin{align*}
\left\| \int_{\Omega} |(T_{g}\rho)(x)|^{2} \, dx \right\|^{1/2} &= \left\| \int_{\Omega} \left| \int_{G} g(x,\theta) \rho(\theta) \, d\theta \right|^{2} \, dx \right\|^{1/2} \\
&\leq \int_{G} \left| \int_{\Omega} \left| \frac{g(x,\theta)}{\|\theta\|^{s}} \right|^{2} \, dx \right|^{1/2} \|\theta\|^{s} |\rho(\theta)| \, d\theta \\
&\leq C_{s} \sqrt{|\Omega|} \|\rho\|_{H^{s}_{1}(G)}.
\end{align*}
This establishes that $\|T_{g} \colon H^{s}_{1}(G) \to L_{2}(\Omega)\| \leq C_{s} \sqrt{|\Omega|}$, thereby completing the proof.
\end{proof}

This result can now be utilized within a Monte Carlo framework. Given $\rho \in H^{s}_{1}(G)$, we define the probability measure as
$$
d\mu(\theta) := \frac{\|\theta\|^{s}|\rho(\theta)|}{\|\rho\|_{H^{s}_{1}(G)}} \, d\theta,\quad \theta\in G.
$$
Employing this probability measure, we obtain the representation
$$
(T_{g}\rho)(x) := \|\rho\|_{H^{s}_{1}(G)} \mathbb{E}_{\mu} \left[ \frac{g(x,\cdot)}{\|\cdot\|^{s}} \text{sgn}(\rho(\cdot)) \right], \quad x \in \Omega.
$$
The following bound is pivotal:

\begin{prop}
  \label{prop:MC-bound}
  Suppose that the activation function~$\sigma$
    obeys~(\ref{eq:kernel-growth}) for some parameter~$s\geq 0$.
For every $\rho \in H^{s}_{1}(G)$, let $u := T_{g}\rho$. Given any set of parameters~$\theta_{1}, \dots, \theta_{n}$, we define the approximation as
\begin{equation}
\label{eq:MC-nn}
u_{n}(x) := \frac{\|\rho\|_{H^{s}_{1}(G)}}{n} \sum_{i=1}^{n} \frac{g(x,\theta_{i})}{\|\theta_{i}\|^{s}} \sgn(\rho(\theta_{i})).
\end{equation}
If the parameters $\theta_{1}, \dots, \theta_{n}$ are independently sampled according to the distribution $\mu$, then
$$
\left( \mathbb{E}_{\mu^{n}} \|u - u_{n}\|_{L_{2}(\Omega)}^{2} \right)^{1/2} \leq C_{s} \|\rho\|_{H^{s}_{1}(G)} \sqrt{\frac{|\Omega|}{n}}.
$$
\end{prop}
\begin{proof}
By Proposition~\ref{prop:integral-L2}, the elements $u$ and $u_{n}$ belong to $L_{2}(\Omega)$. Furthermore, we have the representation
$$
(u - u_{n})(x) = \|\rho\|_{H^{s}_{1}(G)} \frac{1}{n} \sum_{i=1}^{n} \left( \mathbb{E}_{\mu} \left[ \frac{g(x,\cdot)}{\|\cdot\|^{s}} \text{sgn}(\rho(\cdot)) \right] - \frac{g(x,\theta_{i})}{\|\theta_{i}\|^{s}} \text{sgn}(\rho(\theta_{i})) \right).
$$
The above summands are zero-mean random elements in $L_{2}(\Omega)$,
meaning that the variance of the sum is the sum of the
variances. Moreover, the variance of a random element is always less than or equal to its second moment (expected squared $L_{2}$-norm). Therefore, we have
\begin{align*}
\mathbb{E}_{\mu^{n}} \|u - u_{n}\|_{L_{2}(\Omega)}^{2} &= \frac{\|\rho\|_{H^{s}_{1}(G)}^{2}}{n} \mathbb{E}_{\mu^{n}} \left\| \mathbb{E}_{\mu} \left[ \frac{g(x,\cdot)}{\|\cdot\|^{s}} \text{sgn}(\rho(\cdot)) \right] - \frac{g(x,\theta)}{\|\theta\|^{s}} \text{sgn}(\rho(\theta)) \right\|_{L_{2}(\Omega)}^{2} \\
&\leq \frac{\|\rho\|_{H^{s}_{1}(G)}^{2}}{n} \mathbb{E}_{\mu} \left\| \frac{g(x,\cdot)}{\|\cdot\|^{s}} \text{sgn}(\rho(\cdot)) \right\|_{L_{2}(\Omega)}^{2} \\
&\leq \frac{\|\rho\|_{H^{s}_{1}(G)}^{2}}{n} \mathbb{E}_{\mu} \left[ C_{s}^{2} |\Omega| \right] \\
&\leq \frac{|\Omega| C_{s}^{2} \|\rho\|_{H^{s}_{1}(G)}^{2}}{n},
\end{align*}
where the latter bound uses Proposition~\ref{prop:integral-L2}. The proof is complete.
\end{proof}

\begin{remark}
We emphasize that the above modification can be viewed as an {\it importance sampling} strategy in Monte Carlo simulation. By combining this with a suitable {\it stratification} of the parameter space, the convergence rate can be improved to $n^{-1/2 - 1/d}$. For further details, we refer to \cite[Sect.~2]{Xu2020CiCP}.
\end{remark}

To apply the above results to functions $u$ from a spectral Barron space, we need to first identify the kernel~$g$ from the integral operator from~\eqref{eq:integral-operator}, and then to  establish a representation of $u$ in the range of this operator~$T_{g}$, and
we will demonstrate this for the RePU activation function $\sigma_{s}(t) = \max\{0,t\}^{s}$, $t \in \mathbb{R}$, for some integer $s > 0$. 
We follow the calculations in~\cite[Sect.~3.2]{Xu2020CiCP}. 
\par
For the bounded set~$\Omega$ we let~$\mathfrak{T}:= \sup_{x\in\Omega}\norm{x}{\real^d}$.
We consider the parameter set~$G= \set{-1,1}\times [0,\mathfrak{T}] \times \real^d$. The following was shown in the proof of \cite[Lemma 3.4]{Xu2020CiCP} and ~\cite[Eq.~(3.51)]{Xu2020CiCP}:
$$
u(x) - \sum_{\abs{\alpha}\leq s} \frac 1 {\alpha!}D^\alpha u(0) x^\alpha =  \frac{1}{s!}\int_G \lr{z \scalar \omega x-t\norm{\omega}{2}}_+^s \mathfrak s(zt,\omega) \norm{\omega}{2} \abs{\hat u(\omega)}\; d\theta,
$$
with~$\theta:= (z,t,\omega)\in G$.
Above, there is an auxiliary function~$\mathfrak s(zt,\omega)$ (uniformly bounded by one), and~$\hat{u}(\omega)$ denotes the Fourier transform of $u(x)$. The right hand side above should correspond to~$T_g\rho$, up to some multiplicative factor~$1/s!$.
We identify the kernel
$$
g(x,\theta) =\lr{z \scalar \omega x-t\norm{\omega}{2}}_+^s ,\quad x\in\Omega,
$$
and the function~$\rho$ as
\begin{equation}\label{eq:rho-representation}
   \rho(\theta) = \mathfrak s(zt,\omega)\norm{\omega}{2}\abs{\hat u(\omega)}\quad \theta=(z,t,\omega).
\end{equation}

For this representation we deduce the following.
\begin{lemma}
Given some~$s>0$, let the activation function be~$\sigma_s(t) = \max\set{0,t}^s$.
For every~$u\in \Bsp s$ assign~$\rho \in H^{s}_{1}(G)$ as in~\eqref{eq:rho-representation}.
    There is a constant~$C<\infty$ such that 
    $$
    \norm{\rho}{H^s_1(G)}\leq C \norm{u}{\Bspo {s+1}},\quad u\in \Bspo {s+1}.
    $$
\end{lemma}
\begin{proof}
On the parameter set~$G$ we have that~$\norm{\theta}{G} \leq \norm{\omega}{2} +\mathfrak{T} +1$.
To continue we use the inequality
$$
\norm{\theta}{}^{2} \leq \lr{\norm{\omega}{2} + \mathfrak{T} +1}^{2}\leq 2 \lr{\mathfrak {T}
  +1}^{2} \lr{\norm{\omega}{2}^{2} +1}, 
$$
which is easy to verify.
We bound
\begin{align*}
   \norm{\rho}{H^s_1(G)} &=   \int_G \norm{\theta}{}^s
   \abs{s(zt,\omega)}\norm{\omega}{2}
   \abs{\hat u(\omega)}\; d\theta\\
   & \leq \lr{2 \lr{\mathfrak {T}
  +1}^{2}}^{s/2} \int_{\real^d}\int_0^{\mathfrak{T}} \left< \omega\right>^{s}\norm{\omega}{2} \abs{\hat u(\omega)}\; dt \; d\omega\\
   & \leq \lr{2 \lr{\mathfrak {T}
  +1}^{2}}^{s/2} \mathfrak T 
   \int_{\real^d} \left< \omega\right>^{s+1} \abs{\hat u(\omega)}\; d\omega\\
   &= {C(s,\mathfrak{T})} \norm{u}{\Bsp {s+1}},
\end{align*}
for the constant~$C(s,\mathfrak{T}) = \lr{\sqrt 2
  \lr{\mathfrak{T}+1}}^{s}\mathfrak{T}$, which allows us to complete the proof.
\end{proof}
Hence, the bound from Proposition~\ref{prop:MC-bound} specifies as follows:
\begin{corollary}
\label{cor:NN-approx}
Suppose we consider neural networks with the RePU activation function $\sigma_{s}$ with a growth integer $s > 0$. There exists a constant $C < \infty$ such that for $u \in \Bspo {s+1}$, there is a shallow neural network $u_{n}$ with $n$ neurons such that
$$
\|u - u_{n}\|_{L_{2}(\Omega)} \leq C \sqrt{\frac{|\Omega|}{n}} \|u\|_{\Bspo {s+1}},\quad n = 1, 2, \dots.
$$
\end{corollary}

\subsection{Rate and complexity for solving inverse problems}
\label{sec:tikhonov}
The goal of current section is to demonstrate that within the present framework through some form of regularization, a neural network can be found that accurately approximates the unknown solution, and its complexity (width) can be bounded. Recall that we shall further restrict the set of solutions to the set~\eqref{eq:MR}.

Currently, this approach relies on the {\it known smoothness} of the true solution $\udag$, which belongs to the spectral Barron space $\udag \in \Bsp p$, and thus the parameter $p > 0$ is given. In this scenario, we can design a Tikhonov functional as follows:
\begin{equation}
\label{eq:tikhonov}
J_{\lambda}(u) := \|F(u) - \yd\|_Y^2 + \lambda \mathcal{P}(u),\quad u \in \domain(F),
\end{equation}
where the penalty term $\mathcal{P}$ is defined as $\mathcal{P}(u) = \|u\|_{\Bsp p}$. For regularization schemes that incorporate other penalty terms associated with neural networks, we recommend \cite{LiLu2024arxiv}. In current context, following the classic argument presented in \cite{MR1776470}, the regularization parameter $\lambda$ is set to $\lambda := \delta^{2}$, where $\delta$ denotes a measure of the noise present in the observation data.
The main result here is the following.
\begin{theorem}\label{thm:main-tikhonov}
Assume that the mapping~$F\colon \domain(F) \subset \Bsp {-a}\cap \mathcal E^\prime(K) \to Y$ for a compact domain~$K\subset \real^d$ satisfies the link condition from Assumption~\ref{ass:link}, and the noisy data~$\yd$ obey~(\ref{eq:noisy}). Let the parameter~$\lambda$ in the Tikhonov functional~\eqref{eq:tikhonov} be set as~$\lambda:= \delta^2$, and assume that the parameter~$p>0$, which characterizes the smoothness of~$\udag\in \mathcal{M}(R)$, and the parameter~$s$ from the growth condition~\eqref{eq:kernel-growth} satisfy~$p\geq s+1$ and $s>0$.
\par
There exists a constant $C < \infty$ such that, for each $0 < \delta < 1$, there exists a shallow neural network with approximately $\left(\frac{1}{\delta}\right)^{\frac{2p}{a+p}}$ neurons, which achieves the following error bound for any domain~$\Omega \subset K$: 
\begin{equation}
\label{eq:prop-error-p-penalty}
\|\udag - u_{n^{\delta}}\|_{L^{2}(\Omega)} \leq C \delta^{\frac{p}{a+p}}, \quad \text{for all} \ 0 < \delta < 1.
\end{equation}    
\end{theorem}
\begin{proof}
For any compact set~$K\subset \real^d$ the subset~$\mathcal M(R)\subset \Bsp{-a}\cap \mathcal E^\prime(K)$ is relatively compact, see~\cite[Thm.~2.2.3]{MR248435}.  Thus, since under the link condition, Assumption~\ref{ass:link}, the forward mapping~$F\colon \Bsp {-a}\to Y$ is continuous,  a minimizer for the Tikhonov functional (\ref{eq:tikhonov}) exists.

We denote~$\ud$ the minimizer of $J_{\delta^2}$ over the specified set~$\mathcal{M}(R)\cap \mathcal E^\prime(K)$.

Since we assume that the unknown exact solution~$\udag$ belongs to~$\mathcal{M}(R)$, we can derive the following inequality:
\begin{equation}
\label{eq:xd-xdag}
\|F(\ud) - \yd\|_Y^2 + \delta^2 \|\ud\|_{\Bsp{p}} \leq \|F(\udag) - \yd\|_Y^2 + \delta^2 \|\udag\|_{\Bsp{p}}.
\end{equation}
From this inequality, we can obtain two bounds. 
Firstly, we find that~$ \norm{\ud}{\Bsp p} \leq 1+R$, and hence,
both~$\ud$ and~$\udag$ belong to~$\mathcal M(1+R)$.

Secondly, because the right-hand side of the inequality is bounded by $(1 + R)\delta^{2}$, we can derive that:
\[
\|F(\ud) - \yd\|_Y^2 \leq (1 + R)\delta^{2}.
\]
This, in turn, leads to the inequality:
\[
\|F(\ud) - F(\udag)\|_Y \leq (\sqrt{1+R} + 1)\delta. 
\]
By applying the conditional stability estimate from Theorem~\ref{thm:cond-stable} (specifically tailored to the set $\mathcal{M}(1+R)$), we conclude that for a constant~$C_0$ it holds true that
\begin{equation}
\label{eq:xdag-xd-bound}
\|\udag - \ud\|_{L^{2}(\Omega)} \leq \left((\sqrt{1+R} + 1)\delta\right)^{\frac{p}{a+p}} = C_{0} \delta^{\frac{p}{a+p}}.
\end{equation}

In the next step, we approximate $\ud\in\Bsp p$ by a shallow
neural network. To do this, we utilize the well-known~$1/\sqrt
n$-rate, known for spectral Barron spaces, and summarized here in~Corollary~\ref{cor:NN-approx}. By virtue of Remark~\ref{rem:paley} this applies for elements in~$\mathcal M(R)$. Hence, there exists
a shallow neural network, denoted as $u_{n}$, that satisfies:
\[
\|\ud - u_{n}\|_{L^{2}(\Omega)} \leq C_1 \sqrt{\frac{|\Omega|}{n}},
\]
provided that $\ud \in \mathcal M(R)$, see~\eqref{eq:MR},  for $p \geq s+1$.
For this neural network $u_{n}$, we have the following bound relative to the unknown solution $\udag$:
\begin{equation}
\label{eq:err-bound-xn}
\|\udag - u_{n}\|_{L^{2}(\Omega)} \leq C_{0} \delta^{\frac{p}{a+p}} + C_1 \sqrt{\frac{|\Omega|}{n}}.
\end{equation}
To optimize this bound in terms of the number of neurons $n$, we set $n^{\delta}\asymp \left(\frac{1}{\delta}\right)^{\frac{2p}{a+p}}$. Substituting this value into the bound, we finally arrive at:
\begin{equation}
\label{eq:overall-error-p-penalty}
\|\udag - u_{n^{\delta}}\|_{L^{2}(\Omega)} \leq C \delta^{\frac{p}{a+p}},
\end{equation}
uniformly for all $\udag \in  \mathcal M(R)$. The proof is complete.
\end{proof}  
\begin{remark}
In light of the discussion regarding the modulus of continuity, as outlined in Corollary~\ref{cor:modulus}, the rate established in Theorem~\ref{thm:main-tikhonov} in terms of the noise level~$\delta$ seems to be optimal.
Furthermore, it is noteworthy that the number of neurons required is independent of the spatial dimension. This is in stark contrast to analogous problems in the context of Hilbert spaces, where the discretization level (specifically, the number of finite elements) scales as $n \asymp \left(\frac{1}{\delta}\right)^{\frac{d}{a+p}}$, thereby suffering from the curse of dimensionality. For a more detailed exploration of this point, we direct readers to Example 3 in~\cite{MR4748102}.
\end{remark}

\section{Discussion and conclusion}
We characterized spectral Barron spaces as ranges associated with powers of certain weakly sectorial operators. This characterization enabled us to leverage functional calculus to derive a novel interpolation inequality. Interpolation stands as a potent tool in the analysis of operator equations, and its application to operators functioning within scales of spectral Barron spaces appears particularly significant, suggesting a promising area for further investigation.
While the analysis of such spaces on~$\real^d$ seems satisfactory, more structural results for Barron spaces are required, in particular when spaces are considered on bounded domains.
It would be interesting to extend the interpolation to Barron spaces on bounded domains~$\Omega$. This is a typical task in the theory of function spaces. We start with the natural restriction operator (retract)~$ \mathfrak{R} \colon \Bsp s\to \Bspo s$ which assigns~$f$ to its restriction on~$\Omega$.
For instance, in~\cite[Chapt.~3]{MR730762} the following is shown: If for a given range, say~$r \leq s \leq t$, there were a (common) extension operator (coretract)~$\mathfrak{S}\colon \Bspo s\to \Bsp s$, with the property that a) the composition~$\mathfrak{R} \cdot \mathfrak{S}$ is the identity on~$\Bspo s$, and b) $\mathfrak{P}:= \mathfrak{S} \cdot \mathfrak{R}$ is a projection, then interpolation would hold true. For classical function spaces this is the case, at least for~$C^\infty$-domains. For Barron spaces the existence of such coretracts is an open problem. It is interesting to notice, that this problem was already stated in Barron's seminal study~\cite{MR1237720}. 

\par
We formulated inverse (and ill-posed) operator equations in alignment with spectral Barron spaces through a link condition, which subsequently provided a conditional stability estimate for the inverse problem. We demonstrated this approach with three specific cases where such a link condition could be established.
Expanding our understanding to a wider array of inverse problems, particularly (inverse) boundary value problems where measurements are acquired on lower-dimensional boundaries or within lower-dimensional manifolds, is imperative. This endeavor could potentially unlock new insights and methodologies for addressing inverse problems in various scientific and engineering aspects.
\par
The advantage of aligning inverse operator equations with scales of spectral Barron spaces becomes particularly evident when utilizing neural networks as approximate solutions, as this approach enables us to circumvent the curse of dimensionality. The theoretical derivation presented here serves as a foundational framework. However, it would be highly desirable to develop more practical regularization schemes, as these are well-established in the analysis of ill-posed operator equations. Such schemes should be robust in handling noisy data and adaptable to the unknown smoothness of the underlying solutions. By doing so, we can further enhance the practical applicability and performance of inversion techniques.

Finally, we provide some discussion on the numerical implementation.  
Note that in Corollary \ref{cor:NN-approx} and Theorem \ref{thm:main-tikhonov}, the Lipschitz growth property of the activation function, represented by the parameter~$s$, plays a key role.  
Moreover, Theorem~\ref{thm:main-tikhonov} shows that the assumed smoothness parameter~$p$ must satisfy \( p \ge s+1 \).  
From this perspective, it is natural to prefer activation functions with smaller values of $s$, such as ReLU, since lower growth rates allow us to treat broader classes of numerical problems.

Meanwhile, the numerical minimization of the Tikhonov functional (\ref{eq:tikhonov}) requires careful calibration, as the corresponding optimization problem is highly nonlinear and non-convex. In particular, the penalty term involves the Fourier transform, which introduces additional numerical challenges. A typical strategy is to adopt a shallow neural network as an ansatz and optimize the functional with respect to the network parameters. Notably, the required network width can be determined a priori, thereby simplifying the design of the network architecture. We also remark that the greedy algorithm employed for solving forward PDE problems in \cite{MR4728354} could serve as a potential alternative. The development of efficient algorithms to compute the minimizer will be a focus of our future work.

\bibliographystyle{siam}
\bibliography{spectral}

\end{document}